\documentclass{amsart}

\numberwithin{equation}{section}
\usepackage{graphicx}

\newtheorem{thm}{Theorem}[section]
\newtheorem{lem}[thm]{Lemma}

\newtheorem{cond}[thm]{Condition}

\def\cal#1{\fam2#1}

\def\C{{\mathbb C}}

\def\bb{\begin}
\def\be{\begin{equation}}
\def\ee{\end{equation}}
\def\bea{\begin{eqnarray}}
\def\eea{\end{eqnarray}}
\def\beaa{\begin{eqnarray*}}
\def\eeaa{\end{eqnarray*}}

\def\ifl{\iffalse}

\def\bb{\begin}
           \def\ea{\end{array}}
          \def\ec{\end{center}}
     \def\ed{\end{description}}
\def\be{\bb{equation}}        \def\ee{\end{equation}}
\def\bea{\bb{eqnarray}}       \def\eea{\end{eqnarray}}
\def\beaa{\bb{eqnarray*}}     \def\eeaa{\end{eqnarray*}}
 \def\et{\end{thebibliography}}

      \def\e{\varepsilon}

       \def\F{{\cal F}}    
   \def\U{{\cal U}}    \def\X{{\cal X}}    
       
\def\C{\mathcal{C}}

\def\bar2{\doublebar}
            
       \def\nt{\noindent}

\def\qed {\hfill $\Box$\vskip5pt}
\def\Ind {{\rm Ind}}
\def\Int{{\rm Int}}

\def\Sing{{\rm Sing}}

\def\loc{{\rm loc}}

\def\dist{{\rm dist}}
\def\Dom{{\rm Dom}}

\def\Per{{\rm Per}}
\def\Rep{{\rm Rep}}

\def\Orb{{\rm Orb}}
\def\OS{{\rm OrientSh}}
\def\StSh{{\rm StSh}}

\def\bR{{\Bbb R}}

\begin{document}

\title{oriented shadowing property and $\Omega$-stability for vector fields}

\author{Shaobo Gan,  Ming Li, and Sergey B. Tikhomirov}

\address{LMAM and School of Mathematical Sciences, Peking University, Beijing 100871, People's republic of China}
\email{gansb@pku.edu.cn (S. Gan)}

\address{School of Mathematical Sciences and LPMC, Nankai University, Tianjin 300071, People's Republic of China}
\email{limingmath@nankai.edu.cn (M. Li)}

\address{Max Planck Institute for Mathematics in the Sciences, Inselstrasse 22, Leipzig, 04103,
Germany}
\address{Chebyshev Laboratory, Saint-Petersburg State University,
14th line of Vasilievsky island, 29B, Saint-Petersburg, 199178,
Russia} \email{sergey.tikhomirov@gmail.com (S. Tikhomirov)}

\subjclass[2010]{Primary 37C50.}

\keywords{Vector fields; oriented shadowing; $\Omega$-stability.}

\begin{abstract}
We call that a vector field has the oriented shadowing property
if for any $\varepsilon>0$ there is $d>0$ such that each $d$-pseudo
orbit is $\varepsilon$-oriented shadowed by some real orbit. In this paper, we show that the $C^1$-interior of the set of vector
fields with the oriented shadowing property is contained in the set of vector fields with the $\Omega$-stability.
\end{abstract}

\maketitle

\section{Introduction}

The theory of shadowing of approximate trajectories
(pseudo orbits or pseudotrajectories) of dynamical systems is now a well developed part of the global theory of dynamical systems (see, for example,
the monographs \cite{pal, P}). This theory is
closely related to the classical theory of structural stability (the
basic definitions of structural stability and $\Omega$-stability for diffeomorphisms and vector fields can be found, for example, in the monograph \cite{Katok, PilSSBook}).

It is well known that a diffeomorphism  has the shadowing property in
a neighborhood of a hyperbolic set \cite{A, B} and a
structurally stable diffeomorphism (satisfying Axiom A and the strong transversality condition) has the shadowing property on the
whole manifold \cite{R, M, Saw}.
At the same time, the problem of complete description of systems having the
shadowing property seems unsolvable. We have no hope to characterize
systems with the shadowing property in terms of the theory of
structural stability (such as hyperbolicity and transversality)
since the shadowing property is preserved under homeomorphisms of
the phase space (at least in the compact case), while the
above-mentioned properties are not.

The situation changes completely when we pass from the set of smooth
dynamical systems having the shadowing property to its $C^1$-interior. It was shown by Sakai \cite{S}
that the $C^1$-interior of the set of diffeomorphisms with the
shadowing property coincides with the set of structurally stable
diffeomorphisms (see \cite{prs} for some generalizations).

In the present paper we study $C^1$-interior of the set of vector fields with the shadowing property (we call them vector fields with the robust shadowing property). Let us note that the main difference between the shadowing problem for flows and the similar problem for discrete dynamical
systems generated by diffeomorphisms is related to the necessity
of reparametrization of shadowing trajectories in the former
case (see \cite{K, Tho} for the detailed discussion). Another difference arises because of the possibility of accumulation of periodic orbits to a singularity in a robust way.

As in the case of diffeomorphisms, it is well known that a vector field has the shadowing property in a neighborhood of a hyperbolic set \cite{pal, P} and a structurally stable vector field (satisfying Axiom A' and the strong transversality condition) has the shadowing
property on the whole manifold \cite{P, Pilyugin4}.
And as with diffeomorphisms, structural stability is not equivalent to
shadowing.

The structure of the $C^1$-interior of the set of vector fields with the shadowing property is more complicated.

A lot of results show deep connection between $C^1$-interior of the set of vector fields with the shadowing property and structural stability. For vector fields without singularities  such $C^1$-interior consists of only structurally stable vector fields \cite{LS}. On manifolds of dimension 3 and less such $C^1$-interior coincides with the set of structurally stable vector fields \cite{PilTikhDAN, TikhVest}. Vector fields with the robust shadowing property and without
nontransverse intersection stable and unstable manifolds of two
hyperbolic singularities are structurally stable~\cite{PilTikhDAN, PT}.

It is worth to mention studies of the shadowing property for the Lorenz attractor. Komuro \cite{Ko1985} proved that (except very special case) geometric Lorenz flow do not have shadowing property (see also \cite{ARR2013}), however geometric Lorenz flow has parameter-shifted shadowing property for a wide range of parameters \cite{KS}.

At the same time there exists a non-structurally stable vector field on a manifold of dimension 4 which lies in the $C^1$-interior of the set of vector fields with the shadowing property \cite{PT}. The example is a semi-local construction, which consists of two hyperbolic singularities points with complex conjugated eigenvalues and a trajectory of non-transverse intersection of their stable and unstable manifolds. In this example the vector field is not structurally stable due to failure of the strong transversality condition. However it is not clear if one can construct a vector field with the robust shadowing property, which does not satisfy  Axiom A'.

In present work we prove that vector fields with the robust shadowing property are $\Omega$-stable and hence satisfy Axiom A'.

The key role in the proof is played by the star condition (which means that one can not get non-hyperbolic singularities or closed trajectories via a $C^1$ small perturbation, see section \ref{secStar} for the details).
For diffeomorphisms, it is proved in \cite{Ao,Ha,Li,Ma} that the star condition implies the $\Omega$-stability. However, it is not true for vector fields \cite{Di,Gu,LW}. So we have to use additional arguments (Lemmas \ref{lem.weakspace}, \ref{lem.main}) in order to prove $\Omega$-stability.

It is worth to mention another approaches to compare shadowing and structural stability. Analyzing the proofs of the first shadowing results by Anosov
\cite{A} and Bowen \cite{B}, it is easy to see that, in a
neighborhood of a hyperbolic set, the shadowing property is
Lipschitz (and the same holds in the case of structurally stable
diffeomorphisms and vector fields \cite{P}). Moreover structural stability is equivalent to Lipschitz shadowing for diffeomorphisms \cite{pt}, and vector fields \cite{PPT}. Another approach is to consider shadowing and structural stability in $C^1$-generic context. In \cite{AD} it was proved that for so-called ``tame'' diffeomorphisms $C^1$-generically shadowing and structural stability are equivalent. Recently \cite{R2014} it was shown that for vector fields $C^1$-generically shadowing and hyperbolicity are equivalent for isolated sets.

The paper is organised as follows. In section \ref{secDef} we give necessary definitions and state the main result. In section \ref{secStar} we define notion of star flows and formulate necessary statements about them. In section \ref{secProof} we formulate key Lemma \ref{lem.main} and reduce the main theorem to it. In section \ref{sec.linear} we prove some properties of a vector field in a neighborhood of a singularity with a homoclinic connection. In section \ref{secLProof} we complete the proof of Lemma \ref{lem.main}.

\section{Definitions and main results}\label{secDef}

Let $M$ be a compact smooth Riemannian manifold without boundary.
Denote by $\X^1(M)$ the set of $C^1$ vector fields on $M$ endowed with the $C^1$-topology. For a set $A \subset \X^1(M)$ denote $\Int^1(A)$ the interior of $A$ in the $C^1$-topology.

For any $X\in\X^1(M)$, $X$ generates a $C^1$ flow
$$
\phi_t=\phi_{X,t}:M\to M, \ t\in \bR.
$$
Let
$$
\Sing(X)=\{x\in M: X(x)=0\}
$$
be the set of singularities of $X$ and
$$
\Per(X)=\{x\in M\setminus\Sing(X): \exists\ T {\rm \ such\ that\ } \phi_{_T}(x)=x\}
$$
be the set of regular periodic points of $X$. (Here we say a point is {\it regular} if it is not a singularity.)

Denote by
$\Orb(x)=\Orb_X(x)=\phi_{(-\infty,+\infty)}(x)$ the orbit of $x$. And denote by $\Orb^+(x)=\phi_{[0,+\infty)}(x)$, $\Orb^-(x)=\phi_{(-\infty,0]}(x)$ the positive, negative orbit of $x$ respectively.

For any $d>0$, a map $g:\bR\to M$ (not necessarily continuous) is
called a {\it $d$-pseudo orbit}, if
$$
\dist(g(t+\tau),\phi_\tau(g(t)))<d
$$
for all $t\in \bR$ and $\tau\in[0,1]$.

A reparametrization is an increasing homeomorphism $h$ of the line
$\bR$ such that $h(0)=0$; we denote by $\Rep$ the set of all reparametrizations.
For $a>0$, we denote
$$
\Rep(a)=\left\{h\in\Rep:\;\left|\frac{h(t)-h(s)}{t-s}-1\right|<a,\quad
t,s\in\bR,\;t\neq s\right\}.
$$

We say that a vector field $X\in \X^1(M)$ has the {\it (standard) shadowing property} if for any $\varepsilon>0$ there exists a constant $d>0$ such that, for any  $d$-pseudo orbit $g$ of $X$, $g$ is $\varepsilon$-shadowed by a real orbit of $X$, that is,
there exists a point $x\in M$  and a reparametrization $h \in \Rep(\e)$ satisfying
\begin{equation}\label{eqSh}
\dist(g(t),\phi_{h(t)}(x))\le \varepsilon, \quad \mbox{for all $t$}.
\end{equation}
Denote by  $\StSh(M)$ the set of vector fields satisfying the shadowing property.

We say that a vector field $X\in \X^1(M)$ has the {\it oriented shadowing property} if for any $\varepsilon>0$ there exists a constant $d>0$ such that, for any  $d$-pseudo orbit $g$ of $X$, $g$ is $\varepsilon$-oriented shadowed by a real orbit of $X$, that is,
there exists a point $x\in M$  and a reparametrization $h \in \Rep$ satisfying (\ref{eqSh}).
Denote by  $\OS(M)$ the set of vector fields satisfying the oriented shadowing property. Moreover, we say
that a vector field $X$ has the {\it $C^1$-robustly oriented shadowing property} if $X\in \Int^1(\OS(M))$.

Let us note that the standard shadowing property is equivalent to
the strong pseudo orbit tracing property (POTP) in the sense of
Komuro \cite{K}; the oriented shadowing property was called the
normal POTP by Komuro \cite{K} and the POTP for flows by Thomas
\cite{Tho}.

Reparametrizations are essential in the definition of shadowing property, since without reparametrizations vector fields do not have shadowing even in a neighborhood of hyperbolic closed trajectory (see remark after Theorem 1.5.1 in \cite{P}).

Standard and oriented shadowing properties differs only in restrictions on reparametrizations. In case of standard shadowing reparametrization is asked to be close to identity and in case of the oriented shadowing it can be an arbitrarily increasing homeomorphism. Clearly
\begin{equation}\label{Text3.1}
\StSh(M) \subset \OS(M).
\end{equation}
Recently it was shown that the difference in the choice of reparametrization is essential,  so inclusion (\ref{Text3.1}) is strict \cite{TikhOrientNotSt}.
However for vector fields without singularities standard and oriented shadowing properties are equivalent \cite{K}

In the present paper we consider oriented shadowing property.
Below is the main result of this paper.

\begin{thm}\label{thm.MainThm}
Every vector field satisfying the $C^1$-robustly oriented shadowing property is
$\Omega$-stable.
\end{thm}
We would like to note that for manifolds of the dimension $3$ or less \cite{PilTikhDAN, TikhVest} and for vector fields without singularities on manifolds of any dimension \cite{LS} it was proved that $C^1$-robustly oriented shadowing property is equivalent to structural stability. In the present paper we suggest to the reader to have in mind vector fields with singularities on at least 4-dimensional manifolds. In that case there exists essential difference between oriented and standard shadowing properties \cite{TikhOrientNotSt} and certain non-structurally stable examples of vector fields with robust shadowing property exists \cite{PT}.

\section{Star vector fields}\label{secStar}

In the proofs we need some results about star
vector fields. Recall that a vector field $X\in\X^1(M)$ is a {\it star
vector field} on $M$ if $X$ has a $C^1$ neighborhood $\U$ in
$\X^1(M)$ such that, for every $Y\in\U$, every singularity of $Y$
and every periodic orbit of $Y$ is hyperbolic. Denote by
$\X^*(M)$ the set of star vector fields on $M$.

It was proved in \cite{PT} that every vector field satisfying the $C^1$-robustly oriented shadowing property is a star vector field.

\begin{lem}\label{lem.star}
$\Int^1(\OS(M))\subset \X^*(M).$
\end{lem}

We say that a point $x\in M$ is {\it preperiodic} of $X$,  if for any
$C^1$ neighborhood $\U$ of
$X$ in $\X^1(M)$ and any neighborhood $U$ of $x$ in $M$, there exists $Y\in\U$ and $y\in U$ such that $y$
is a regular periodic point of $Y$. Denote by
$\Per_*(X)$ the set of preperiodic points of $X$. We will use the
following result which is proved in \cite{GW}.

\begin{thm}\label{thm.GW}
Let $X\in \X^*(M)$. If
$\Sing(X)\cap\Per_*(X)=\emptyset$, then $X$ is $\Omega$-stable.
\end{thm}

We also need some results about the dominated splitting in the tangent space of singularities. Let $\sigma$ be a hyperbolic singularity of $X$. Denote by
$$
{\rm Re}(\lambda_s)\le \cdots \le {\rm Re}(\lambda_2)\le{\rm
Re}(\lambda_1)<0<{\rm Re}(\gamma_1)\le{\rm
Re}(\gamma_2)\le\cdots\le{\rm Re}(\gamma_u),
$$
the eigenvalues of ${\rm D}X(\sigma)$.
 The {\it saddle value} of $\sigma$ is
$$
{\rm SV}(\sigma)={\rm Re}(\lambda_1)+{\rm Re}(\gamma_1).
$$

We write $\Ind(\sigma)$ the {\it index} of a hyperbolic singularity $\sigma\in \Sing(X)$ which is the dimension of the stable
manifold of $\sigma$. We write $\Ind(p)$ the {\it index} of a regular hyperbolic periodic point $p\in\Per(X)$ which is the dimension of the strong
stable manifold of $p$.

Recall that a {\it homoclinic connection} $\Gamma$ of a singularity $\sigma$ is the closure of a orbit of a regular point which is contained in both the stable and the unstable manifolds of $\sigma$.
The following lemma is a simplified version of results in \cite{SS}.

\begin{lem}\label{lem.shilnikov}
Let $\sigma\in\Sing(X)$ be a singularity of vector field $X\in \X^1(X)$ exhibiting a homoclinic connection $\Gamma$. If ${\rm SV}(\sigma)\ge 0$, then for any $C^1$ neighborhood $\U$ of $X$ in $\X^1(M)$ and any neighborhood $U$ of $\Gamma$ in $M$, there exists $Y\in\U$ and $p\in\Per(Y)$ such that $\Orb_Y(p)\subset U$ and $\Ind(p)=\Ind(\sigma)-1$.
\end{lem}

By using the same argument in the proof of \cite[Lemma 4.1]{LGW}, we can get:

\begin{lem}\label{lem.domspl}
Let $X\in\X^*(M)$ and $\sigma\in\Sing(X)$ be a singularity of $X$. If there exists a integer $1\le I\le \Ind(\sigma)-1$ such that, for any $C^1$ neighborhood $\U$ of $X$ in $\X^*(M)$ and any neighborhood $U$ of $\sigma$ in $M$, there exists $Y\in\U$ and $p\in U\cap\Per(Y)$ with $\Ind(p)=I$. Then $E^s_\sigma$ splits into a dominated splitting
$$
E^s_\sigma=E^{ss}_\sigma\oplus E^c_\sigma,
$$
where $\dim E^{ss}_\sigma=I$.
\end{lem}

Combining Lemma~\ref{lem.shilnikov} with Lemma~\ref{lem.domspl}, we obtain directly the following lemma about  singularities of star vector fields exhibiting a homoclinic connection.

\begin{lem}\label{lem.weakspace}
Let $X\in\X^*(M)$ be a star vector field and $\sigma\in\Sing(X)$ be a singularity of $X$ exhibiting a homoclinic connection.
\begin{itemize}
  \item If ${\rm SV}(\sigma)\ge 0$ and $\dim E^s_\sigma\ge 2$, then $E^s_\sigma$ splits into a dominated splitting
$$
E^s_\sigma=E^{ss}_\sigma\oplus E^c_\sigma,
$$
where $\dim E^{c}_\sigma=1$;
  \item If ${\rm SV}(\sigma)\le 0$ and $\dim E^u_\sigma\ge 2$, then $E^u_\sigma$ splits into a dominated splitting
$$
E^u_\sigma=E^{c}_\sigma\oplus E^{uu}_\sigma,
$$
where $\dim E^{c}_\sigma=1$.
\end{itemize}
\end{lem}

\section{Proof of Theorem~\ref{thm.MainThm}}\label{secProof}

The key step of the proof is the following lemma, which will be proved in the next sections.

\begin{lem}\label{lem.main}
Let $X\in \X^*(M)$. If there exists a singularity $\sigma\in \Sing(X)$ exhibiting a homoclinic connection, then
$X\not\in\Int^1(\OS(M))$.
\end{lem}

To create a homoclinic connection by $C^1$ perturbations, we need the following uniform $C^1$ connecting lemma.

\begin{thm}\cite{W}\label{thm.connecting}
Let $X\in \X^1(M)$. For any $C^1$ neighborhood $\U \subset \X^1(M)$ of $X$ and any point $z\in M$ which is neither singular nor
periodic of $X$, there exist three numbers $\rho>1$, $T>1$ and $\delta_0>0$, together with a $C^1$ neighborhood $\U_1\subset \U$ of $X$ such that for any $X_1\in \U_1$, any $0<\delta<\delta_0$ and any two points
$x$, $y$ outside the tube $\Delta_{X_1,z}= \cup_{t\in [0,T]}B(\phi_{X_1,t}(z), \delta)$, if the positive $X_1$-orbit of $x$
and the negative $X_1$-orbit of $y$ both hit $B(z, \delta/\rho)$, then there exists $Y\in \U$
with $Y = X_1$ outside $\Delta_{X_1,z}$ such that $y$ is on the positive $Y$-orbit of $x$.
\end{thm}

As a classical application of the connecting lemma, we have the following.

\begin{lem}\label{lem.pertubation}
Let $X\in \X^1(M)$ and $\sigma\in \Sing(X)$ be a hyperbolic singularity of $X$ which is preperiodic. Then for any $C^1$ neighborhood $\U \subset \X^1(M)$ of $X$, there exists $Y\in \U$ such that $\sigma_{_Y}\in \Sing(Y)$ exhibiting a homoclinic connection, where $\sigma_{_Y}$ is the continuation of $\sigma$.
\end{lem}
\begin{proof}
Since $\sigma$ is preperiodic, there exists a sequence of vector fields $\{X_n \in \X^1(M)\}$ together with a sequence of periodic points $\{o_n\in \Per(X_n)\}$ such that $X_n$ $C^1$-approximate to $X$ and $o_n$ approximate to $\sigma$.
By using a typical argument and taking converging subsequence if necessary, we have that there exist $p_n, q_n \in \Orb_{X_n}(o_n)$ and two points $a\not=\sigma$ and $b\not=\sigma$ on the local stable manifold $W^{s,X}_{\loc}(\sigma)$ and local unstable manifold $W^{u,X}_{\loc}(\sigma)$ of $\sigma$ respectively such that $p_n$ approximate to $a$ and $q_n$ approximate to $b$.

For a $C^1$ neighborhood $\U \subset \X^1(M)$ of $X$ and the point $a$, by Theorem~\ref{thm.connecting}, there exist numbers $\rho_a>1$, $T_a>1$, $\delta_{a,0}>0$ and a $C^1$ neighborhood $\U_{a,1}\subset \U$ of $X$ with the property of the connecting lemma. Similarly, for the neighborhood $\U_{a,1}$ of $X$ and the point $b$, there exist numbers $\rho_b>1$, $T_b>1$, $\delta_{b,0}>0$ and a $C^1$ neighborhood $\U_{b,1}\subset \U_{a,1}$ of $X$ with the property of the connecting lemma. Choose $0<\delta<\min\{\delta_{a,0},\delta_{b,0}\}$ small enough such that the tubes $\Delta_{X,a}= \cup_{t\in [0,T_a]}B(\phi_{X,t}(a), \delta)$ and $\Delta_{X,b}= \cup_{t\in [0,T]}B(\phi_{X,t}(b), \delta)$ satisfy
\begin{equation}\label{property}
\overline{\Delta_{X,a}} \cap \overline{\Delta_{X,b}}=\emptyset,\ \ \overline{\Delta_{X,a}} \cap \overline{W^{u,X}_{\loc}(\sigma)}=\emptyset,\ \ {\rm and} \ \ \overline{\Delta_{X,b}} \cap \overline{W^{s,X}_{\loc}(\sigma)}=\emptyset.
\end{equation}

By the Invariant Manifold Theorem, there is a $C^1$ neighborhood $\U_1\subset \U_{b,1}$ of $X$ such that property (\ref{property}) holds for any $X'\in \U_1$. Decrease $\U_1$ if necessary, we may also assume that for any $X'\in \U_1$ the local stable manifold $W^{s,X'}_{\loc}(\sigma_{_{X'}})$ of $\sigma_{_{X'}}$ hit $B(a, \delta/\rho)$ and the local unstable manifold $W^{u,X'}_{\loc}(\sigma_{_{X'}})$ of $\sigma_{_{X'}}$ hit $B(b, \delta/\rho)$.

Take $n$ big enough such that $X_n \in \U_1$ and $p_n\in B(a, \delta/\rho)$, $q_n \in B(b, \delta/\rho)$.
Thus we can take two points $a' \in W^{s,X_n}_{\loc}(\sigma_{_{X_n}}) \cap B(a, \delta/\rho)$ and $b'\in W^{u,X_n}_{\loc}(\sigma_{_{X_n}}) \cap B(b, \delta/\rho)$. Choose two points $x\in \Orb_{X_n}^-(b')\setminus \Delta_{X_n,b}$ and $y\in \Orb_{X_n}^+(a')\setminus \Delta_{X_n,a}$. Note that $x\in W^{u,X_n}_{\loc}(\sigma_{_{X_n}})$ and $y\in W^{s,X_n}_{\loc}(\sigma_{_{X_n}})$.

Since the positive $X_n$-orbit of $x$
and the negative $X_n$-orbit of $p_n$ both hit $B(b, \delta/\rho)$ (at $b'$ and $q_n$), there exists $Z\in \U_{a,1}$
with $Z = X_n$ outside $\Delta_{X_n,b}$ such that $p_n\in \Orb_{Z}^+(x)$. Now we have the positive $Z$-orbit of $x$
and the negative $Z$-orbit of $y$ both hit $B(a, \delta/\rho)$ (at $p_n$ and $a'$). Thus we can use the connecting lemma again and get a homoclinic connection of $\sigma_{_Y}$ for some $Y\in \U$.
\end{proof}

Now Theorem \ref{thm.MainThm} is a consequence of Lemmas~\ref{lem.star}, \ref{lem.main}, \ref{lem.pertubation} and Theorem~\ref{thm.GW}. Indeed,

\begin{proof}[Proof of Theorem~\ref{thm.MainThm}]
On the contrary, suppose that there exists a vector field  $X\in \Int^1(\OS(M))$ satisfying the $C^1$-robustly oriented shadowing property which is not $\Omega$-stable. By Lemma~\ref{lem.star} we know that $X\in\X^*(M)$ is a star vector field. Since $X$ is not $\Omega$-stable, we have that  $\Sing(X)\cap\Per_*(X)\ne\emptyset$ according to Theorem~\ref{thm.GW}.
Suppose that $\sigma\in \Sing(X)\cap\Per_*(X)$ is a preperiodic singularity.

Then by Lemma~\ref{lem.pertubation}, there exists a vector field $Y$ arbitrarily $C^1$ close to $X$ such that $\sigma_{_Y}\in \Sing(Y)$ exhibiting a homoclinic connection, where $\sigma_{_Y}$ is the continuation of $\sigma$.
Note that we have $Y\in \Int^1(\OS(M))$ when $Y$ close enough to $X$. It contradicts with Lemma~\ref{lem.main}.
\end{proof}

The rest part of the paper is devoted to the proof of Lemma \ref{lem.main}. It follows the strategy similar to \cite[Section 2, Case (B1)]{PT}.

\section{Local linear model of homoclinic connection}\label{sec.linear}
 In this section, we will consider a local linear model of a singularity exhibiting a homoclinic connection. Let $\sigma\in\Sing(X)$ be a saddle type hyperbolic singularity of $X$ and
$$
\Gamma\subset W^s(\sigma)\cap W^u(\sigma)
$$
be a homoclinic connection of $\sigma$. Denote by $s=\dim E^s_\sigma$, $u=\dim E^u_\sigma$ the dimension of contracting and expending subspaces of $\sigma$.

In this section we assume that $X$ satisfies the following
\begin{cond}\label{cond1}
\begin{itemize}
  \item $X$ is linear in a small neighborhood $U(\sigma)$ of
$\sigma$;
  \item  $s\ge 2$ and there exists a  dominated splitting $E^s_\sigma=E^{ss}_\sigma\oplus E^c_\sigma$, where $\dim E^c_\sigma=1$;
  \item $\{E^{ss}_\sigma, E^c_\sigma, E^u_\sigma\}$ are
mutually orthogonal;
  \item $\Gamma\cap W^{ss}(\sigma)=\{\sigma\}$, where $W^{ss}(\sigma)$ is the strong stable manifold which is tangent to $E^{ss}_\sigma$.
\end{itemize}
\end{cond}

We introduce a orthogonal coordinate system $(x^{ss},x^c,x^u)$ with respect to the splitting
$$
T_\sigma M=E^{ss}_\sigma\oplus E^c_\sigma\oplus E^u_\sigma
$$
 in the neighborhood $U(\sigma)$, where  $x^{ss},x^c,x^u$ are the coordinates of $x\in U(\sigma)$ in $E^{ss}_\sigma, E^c_\sigma, E^u_\sigma$ respectively.

We take two points $O_s=(O_s^{ss},O_s^c,0), O_u=(0,0,O_u^u) \in\Gamma\cap
U(\sigma)\setminus\{\sigma\}$ contained in $E^s_\sigma$ and $E^u_\sigma$
respectively. Note that $O_s^c\ne 0$ since $\Gamma\cap W^{ss}(\sigma)=\{\sigma\}$.
In the neighborhood $U(\sigma)$, we choose a small codimension $1$ cross section $\Sigma_{s}$ containing $O_s$ which is orthogonal to $E^c_\sigma$ and another small codimension $1$ cross section $\Sigma_{u}$ containing $O_u$ which is parallel to $E^s_\sigma$.

 Denote by
$Q:\Sigma_u\to\Sigma_s$ the Poincar\'e map from
$\Sigma_u$ to $\Sigma_s$, and $\tau_{_Q}(x)$ the
minimal positive $t$ such that
$\phi_{t}(x)=Q(x)$. Reduce $\Sigma_u$ if necessary, we may assume that
$$
Q(\overline{\Sigma_u})\subset \Sigma_s.
$$
Thus $\tau_{_Q}(x)$ is bounded for $x\in\Sigma_u$.

Let
$$
L_s=\{x\in\Sigma_s: x^{u}=0\}
$$
be a $(s-1)$-dimensional disc in $\Sigma_s$ and
$$
L_u=\{x\in\Sigma_u: x^{ss}=0\}
$$
be a $u$-dimensional disc in $\Sigma_u$.

We point that $L_s=W^s_{\loc}(\sigma)\cap \Sigma_s$, where $W^s_{\loc}(\sigma)$ is the local stable manifold of $\sigma$. Since
$$
\dim \Sigma_s=\dim M-1=s+u-1.
$$
In what follows we additionally assume
\begin{cond}\label{cond2}
$Q(L_u)$ is transverse to $L_s$ at $O_s$ in $\Sigma_s$.
\end{cond}
For any $\eta\ge 0$, we define a cone in $\Sigma_u$:
$$
\C^u_\eta=\{x\in\Sigma_u: |x^{ss}|\le\eta|x^{c}|\}.
$$

We have the following lemma. (See Figure~\ref{fig.cone}.)


\begin{figure}[ht]
\begin{center}
\begin{minipage}[b]{0.49\linewidth}
\centering
\includegraphics[width=\textwidth]{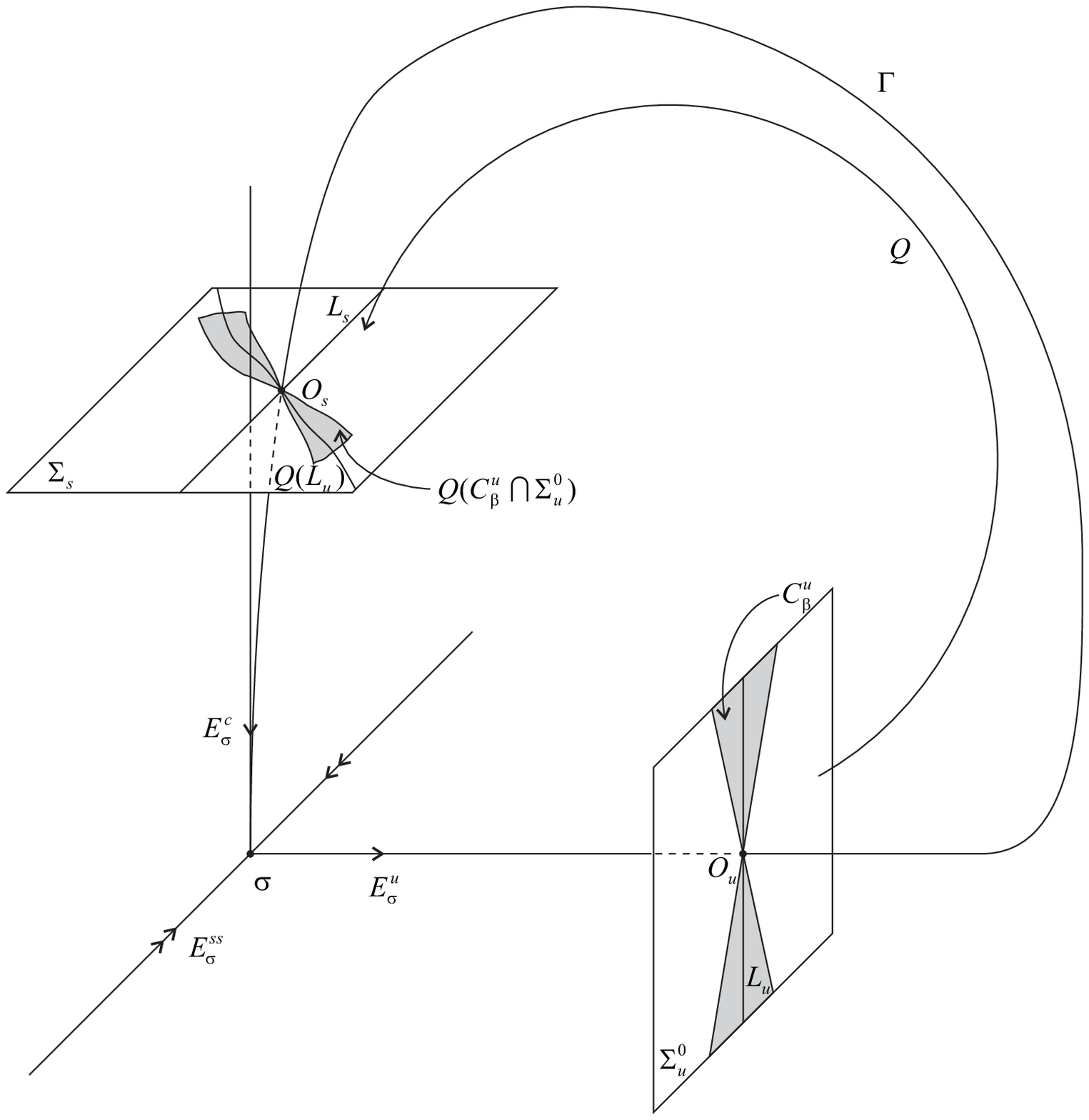}
\caption{}\label{fig.cone}
\end{minipage}
\begin{minipage}[b]{0.41\linewidth}
\centering
\includegraphics[width=\textwidth]{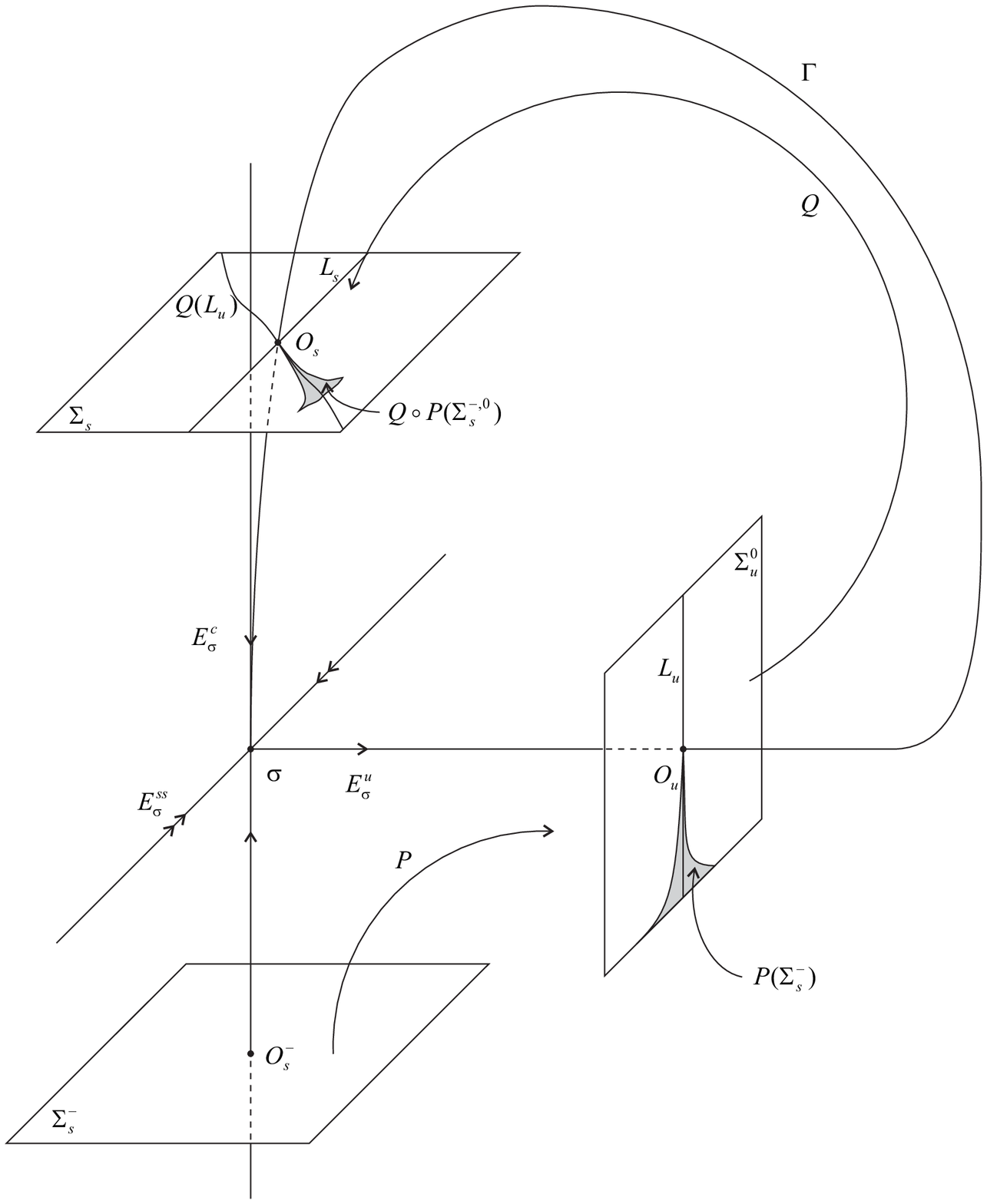}
\caption{}\label{fig.contra}
\end{minipage}
\end{center}
\end{figure}

\begin{lem}\label{lem.cone}
There exists a positive constant $\beta>0$ and a neighborhood $\Sigma_u^0\subset\Sigma_u$ of $O_u$ in $\Sigma_u$ such that
$$
Q(\C^u_\beta\cap\Sigma_u^0)\cap L_s=\{O_s\}.
$$
\end{lem}
\begin{proof}
Note that $L_u=\C^u_0$ and $Q(L_u)$ is transverse to $L_s$ at $O_s$ in $\Sigma_s$. Since $\tau_{_Q}(x)$ is bounded for $x\in\Sigma_u$, we can get $\beta>0$ and $\Sigma_u^0\subset\Sigma_u$ such that $Q(\C^u_\beta\cap\Sigma_u^0)\cap L_s=\{O_s\}$ by the continuity.
\end{proof}

Denote $O_s^-=(0,-O_s^c,0)$, and let $\Sigma_{s}^-$ be a small codimension $1$ cross section in $U(\sigma)$ which is orthogonal to $E^c_\sigma$ at $O_s^-$.   Denote by
$P:\Dom(P)\subset\Sigma_s^-\to\Sigma_u^0$ the Poincar\'e map from $\Sigma_s^-$ to
$\Sigma_u^0$ in $U(\sigma)$, where
$$
\Dom(P)=\{x\in\Sigma_s^-: \exists\ T>0 \ {\rm such\ that\ }\phi_T(x)\in \Sigma^0_u \ {\rm and \ } \phi_{[0,T]}(x)\subset U(\sigma)\}
$$
the domain of $P$.  Denote by $\tau_{_P}(x)$ the
minimal positive $t$ such that $\phi_{t}(x)=P(x)$.

Since the contraction of the strong stable subspace $E_\sigma^{ss}$ is stronger than that of the $1$-dimensional subspace $E_\sigma^c$. For any point $x\in \Dom(P)$, $P(x)$ tends to the direction $E_\sigma^c$ (that is, $|x^{ss}|/|x^c|$ decreases by $P$). The time $\tau_{_P}(x)$ form $\Sigma_s^-$ to $\Sigma_u^0$ is arbitrarily large when the point $x$ is arbitrarily close to $W^s_\loc(\sigma)\cap \Sigma_s^-$. It implies that $|(P(x))^{ss}|/|(P(x))^c|$ tends to $0$ as $x$ tends to $W^s_\loc(\sigma)\cap \Sigma_s^-$. Thus the image of the Poincar\'e map $P$ has a ``sharp'' at the point $O_u$.
(See Figure~\ref{fig.contra}.)
An example in dimension $3$ is the geometric Lorenz attractor. Indeed, we have the following lemma.


\begin{lem}\label{lem.conecontr}
There exists a neighborhood $\Sigma_s^{-,0}\subset\Sigma_s^-$ of $O_s^-$ in $\Sigma_s^-$ such that, for any $x\in \Dom(P)\cap\Sigma_s^{-,0}$ we have
$$
P(x)\in \C^u_\beta\cap\Sigma_u^0.
$$
\end{lem}
\begin{proof}
Denote by
$$
a = \inf\{|x^u|: x\in\Sigma^0_u\}>0
$$
the infimum of the $E^u_\sigma$-coordinate of the points in $\Sigma^0_u$ and
$$
\gamma = \max\{{\rm Re}(\alpha):  \alpha\ {\rm is\ a\ eigenvalue\ of\ D}X(\sigma)\}>0
$$
the maximum expansion in $E^u_\sigma$.

Then for any $x\in \Dom(P)$ (note that $x^u\ne 0$) we have
$$
|x^u|e^{\gamma\tau_{_P}(x)}\ge  |(P(x))^u|\ge a.
$$

It implies that
$$
\tau_{_P}(x)\to +\infty\ \ {\rm as}\ \ |x^{u}|\to 0.
$$
On the other hand, since $E^{ss}_\sigma$ is the strong contracting subspace, we have
$$
 \frac{|(P(x))^{ss}|}{|(P(x))^{c}|}  \left / \frac{|x^{ss}|}{|O_s^{c}|} \right. =\frac{|(P(x))^{ss}|}{|x^{ss}|}  \left / \frac{|(P(x))^{c}|}{|O_s^{c}|} \right. \to 0 \ \ {\rm as} \ \tau_{_P}(x)\to +\infty.
$$

Hence $P(x)\in \C^u_\beta$ as $|x^{u}|$ is small enough. Thus there exists a neighborhood $\Sigma_s^{-,0}\subset\Sigma_s^-$ of $O_s^-$ in $\Sigma_s^-$ such that
$$P(x)\in \C^u_\beta\cap\Sigma_u^0$$
for any $x\in \Sigma_s^{-,0}\cap \Dom(P)$.
\end{proof}

Combined Lemma~\ref{lem.cone} and Lemma~\ref{lem.conecontr},  we can get the following lemma immediately.

\begin{lem}\label{lem.notinstable}
For any point $x\in \Sigma_s^{-,0}\cap \Dom(P)$, we have
$$Q\circ P(x)\not\in L^s.$$
\end{lem}
\qed

\section{The proof of Lemma~\ref{lem.main}}\label{secLProof}

Suppose on the contrary that there exists a star vector field $X\in \Int^1(\OS(M))$ which  has a singularity $\sigma\in\Sing(X)$ exhibiting a homoclinic connection.

Up to an arbitrarily $C^1$ small perturbation, we may assume that $X$ is linear in a small neighborhood $U_r(\sigma)$ of
$\sigma$ on a proper chart, still exhibits a homoclinic connection $\Gamma\subset W^s(\sigma)\cap W^u(\sigma)$ (see \cite{PT} for more details on the perturbations). Without loss of generality, we can assume that ${\rm SV}(\sigma)\ge 0$.

We consider the following two cases:

\smallskip

\nt Case 1. $\dim E^s_\sigma=1$.

\smallskip

In this case, take $\varepsilon=r/10$, then there exists $d>0$ such that every $d$-pseudo orbit can be $\varepsilon$-oriented shadowed by a real orbit of $X$.

Take $p\in W^s(\sigma)\setminus \Gamma$ and $q\in W^u_{\loc}(\sigma)\cap \Gamma$ in a small neighborhood of $\sigma$ such that the map
$$
g(t)=\left \{
\begin{array}{ll}
           \phi_t(p), & t\le 0 ; \\
           \phi_t(q), & t>0.
        \end{array}
 \right.
$$
is a $d$-pseudo orbit. Thus $g$ is $\varepsilon$-oriented shadowed by a real orbit $\Orb(x)$. Note that $q\in W^s(\sigma)$ implies that $x\in W^s(\sigma)$. But since $\dim E^s_\sigma=1$ we have $x\in \Gamma$. It is a contradiction.

\smallskip

\nt Case 2. $\dim E^s_\sigma\ge 2$.

\smallskip

In this case, by Lemma~\ref{lem.weakspace} we know that there exists a dominated splitting
$$
E^s_\sigma=E^{ss}_\sigma\oplus E^c_\sigma,
$$
where $\dim E^c_\sigma=1$.
Changing the Riemannian metric if necessary,  we assume that $\{E^{ss}_\sigma, E^c_\sigma, E^u_\sigma\}$ are
mutually orthogonal.

By an arbitrarily small perturbation we can find a vector field $X^* \in \Int^1(\OS)$, satisfying Conditions \ref{cond1}, \ref{cond2}. For simplifuing the notation we denote it by $X$ again.

Below we are using the notation from section \ref{sec.linear}.


Let $\beta>0$ and $\Sigma_s^{-,0}\subset \Sigma_s^-$ be given by Lemma~\ref{lem.cone} and Lemma~\ref{lem.conecontr}. Then for any point $x\in \Sigma_s^{-,0}\cap \Dom(P)$, we have
$$
Q\circ P(x)\not\in L^s.
$$

Take $0<\varepsilon<\min\{|O_s^c|/10,|O_u^u|/10\}$ small enough such that, for any point $x$ in $U_\varepsilon(O_s)$ ($U_\varepsilon(O_s^-)$, $U_\varepsilon(O_u)$), the component of $\Orb(x)$ in $U_\varepsilon(O_s)$ ($U_\varepsilon(O_s^-)$, $U_\varepsilon(O_u)$) containing $x$ intersects $\Sigma_s$ ($\Sigma_s^{-,0}$, $\Sigma_u^0$).

Let $d>0$ be the parameter given by the oriented shadowing property according to $\varepsilon$. Since $O_s^-\in W^s_{\loc}(\sigma)$ and $O_u\in W^u_{\loc}(\sigma)$, there are $p\in \Orb^+(O_s^-)$ and $q\in \Orb^-(O_u)$ in a small neighborhood of $\sigma$ such that the map
$$
g(t)=\left \{
\begin{array}{ll}
           \phi_t(p), & t\le 0 ; \\
           \phi_t(q), & t>0.
        \end{array}
 \right.
$$
is a $d$-pseudo orbit. Then there exists a point $x$  and an increasing homeomorphism
 $h(t)$ of the real line such that
$$
\dist(g(t),\phi_{h(t)}(x))\le \varepsilon
$$
for all $t$.

Denote by $t_1$, $t_2$ and $t_3$ the  parameters satisfying $g(t_1)=O_s^-$, $g(t_2)=O_u$ and $g(t_3)=O_s$.
According to the choice of $\varepsilon$, the component of $\Orb(\phi_{h(t_1)}(x))$ ($\Orb(\phi_{h(t_2)}(x))$, $\Orb(\phi_{h(t_3)}(x))$) in $U_\varepsilon(O_s^-)$ ($U_\varepsilon(O_u)$, $U_\varepsilon(O_s)$) would intersect $\Sigma_s^{-,0}$ ($\Sigma_u^{0}$, $\Sigma_s$). Denote by $\phi_{h(\widetilde{t_1})}(x)$, $\phi_{h(\widetilde{t_2})}(x)$ and $\phi_{h(\widetilde{t_3})}(x)$ the intersection points respectively.

Since $\Orb(x)$ $\varepsilon$-oriented shadows $g$, we have that
$$
\phi_{h(\widetilde{t_2})}(x)=P(\phi_{h(\widetilde{t_1})}(x)) \ \ {\rm and}\ \ \phi_{h(\widetilde{t_3})}(x)=Q(\phi_{h(\widetilde{t_2})}(x)).
$$

Thus by Lemma~\ref{lem.notinstable}, we can get
$$
\phi_{h(\widetilde{t_3})}(x)\not\in L^s.
$$
It implies that $\phi_{h(\widetilde{t_3})}(x)\not\in W^s_{\loc}(\sigma)$.

But on the other hand, since $O_s=g(t_3)\in W^s_{\loc}(\sigma)$ we have
$$
\Orb^+(\phi_{h(t_3)}(x))\subset U_r(\sigma).
$$
Thus
$$
\Orb^+(\phi_{h(\widetilde{t_3})}(x))\subset U_r(\sigma).
$$
It implies $\phi_{h(\widetilde{t_3})}(x)\in W^s_{\loc}(\sigma)$.

This contradiction proves this theorem.
\qed

\section*{Acknowledgment}

Shaobo Gan is supported by 973 project 2011CB808002, NSFC 11025101 and 11231001.
Ming Li is partially supported by the National Science Foundation of China (Grant No. 11201244),
the Specialized Research Fund for the Doctoral Program of Higher Education of China (SRFDP, Grant No. 20120031120024),
and the LPMC of Nankai University.
Sergey Tikhomirov is partially supported by Chebyshev Laboratory (Department of Mathematics and Mechanics, St. Petersburg State University)  under RF Government grant 11.G34.31.0026, JSC ``Gazprom neft'', by the Saint-Petersburg State University research grant 6.38.223.2014 and by  the German-Russian
Interdisciplinary Science Center (G-RISC) funded by the German Federal
Foreign Office via the German Academic Exchange Service (DAAD).

\bibliographystyle{amsplain}

\end{document}